\documentclass[11pt]{amsart}

\usepackage{amsmath, amsthm, amssymb}

\calclayout

\usepackage{graphicx}
\usepackage{tikz-cd}
\usepackage{tikz}
\usetikzlibrary{positioning}
\usetikzlibrary{hobby}
\usepackage{mathtools}

\newcommand{\N}{\mathbb{N}}

\newcommand{\R}{\mathbb{R}}

\newtheorem{thm}{Theorem}[section]

\newtheorem{ex}[thm]{Example}
\newtheorem{cor}[thm]{Corollary}

\newtheorem{defn}[thm]{Definition}
\newtheorem{rmrk}[thm]{Remark}

\DeclareMathOperator{\Vol}{Vol}

\begin{document}

\title{Compactness of Sequences of Warped Product Length Spaces}

\author{Brian Allen}
\author{Bryan Sanchez}
\author{Yahaira Torres}
\address{Lehman College, CUNY}

\maketitle

\begin{abstract}
 If we consider a sequence of warped product length spaces, what conditions on the sequence of warping functions implies compactness of the sequence of distance functions? In particular, we want to know when a subsequence converges to a well defined metric space on the same manifold with the same topology. What conditions on the sequence of warping functions implies Lipschitz bounds for the sequence of distance functions and/or the limiting distance function? In this paper we give answers to both of these questions as well as many examples which elucidate the theorems and show that our hypotheses are necessary.
\end{abstract}

\section{Introduction}

When studying sequences of Riemannian manifolds or length spaces, it is often helpful to have compactness theorems which tell you that at least there is a subsequence which converges. Then the goal is usually to show that the subsequence converges to a particular desired limit space, which will often yield convergence of the entire sequence to that limit space. If we consider a sequence of warped product length spaces then the goal of proving a compactness theorem can be phrased in terms of the sequence of warping functions. What conditions on the sequence of warping functions implies compactness of the sequence of distance functions? In particular, we want to know when a subsequence converges to a well defined metric space on the same manifold with the same topology. What conditions on the sequence of warping functions implies Lipschitz bounds for the sequence of distance functions and/or the limiting distance function? In this paper, we give answers to these questions as well as many examples which elucidate the theorems and show that our hypotheses are necessary.

Similar questions were first studied by B. Allen and C. Sormani \cite{Allen-Sormani} where their goal was to deduce conditions on a sequence of warping functions which guarantees convergence in the uniform sense to a particular warped product length space. Many illuminating examples were given and expanded on in \cite{Allen-Sormani-2} where conformal examples were studied. This investigation then culminated in the proof of the Volume Above Distance Below (VADB) Theorem by B. Allen, R. Perales, and C. Sormani \cite{Allen-Perales-Sormani-VADB} for general sequences of Riemannian manifolds. To give the statement, let $(M,g_j)$ be a sequence of continuous Riemannian manifolds with bounded diameter and $(M,g_{\infty})$ a smooth Riemannian manifold. Then the VADB Theorem says that if $\Vol(M,g_j) \rightarrow \Vol(N,g_{\infty})$ and $g_j \ge \left(1-C_j\right)g_{\infty}$ where $C_j \rightarrow 0$, then $(M,g_j)$ converges to $(M,g_{\infty})$ in the Sormani-Wenger Intrinsic Flat sense. We also note that L.-H. Huang, D. Lee, and C. Sormani \cite{HLS} establish a compactness theorem in the uniform, Gromov-Hausdorff, and Sormani-Wenger Intrinsic Flat sense for metric spaces under the assumption of uniform Lipschitz bounds. Our results suggest that one can prove uniform and Gromov-Hausdorff compactness in general under weaker hypotheses.

The VADB theorem has been used to prove stability of the Llarull theorem by B. Allen, E. Bryden, and D. Kazaras \cite{ABKLLarull} in dimension three and by S. Hirsch and Z. Yang \cite{HZ} in all dimensions. The VADB has also been used to study the conformal case of the scalar torus stability conjecture by Allen \cite{Allen-Conformal-Torus}, Chu and Lee  \cite{Chu-Man-Chun22}, the graph case of Question scalar torus stability by  Pacheco,  Ketterer and  Perales \cite{PKP19},  and the graph case of positive mass theorem stability by Sormani,  Huang, and Lee \cite{HLS}, Allen and Perales \cite{Allen-Perales}, Perales, Huang and Lee \cite{Huang-Lee-Perales} in the asymptotically Euclidean case, and Pacheco,  Graf and Perales \cite{PGPAsymHyp} in the asymptotically hyperbolic case. The development of the VADB theorem would not have been possible without the insights gleaned from many warped product and conformal examples studied by B. Allen and C. Sormani \cite{Allen-Sormani, Allen-Sormani-2} which demonstrates the importance of studying warped product length spaces in order to better understand compactness in general, as we do in this paper.

One important conjecture we would like to understand better is the scalar curvature compactness conjecture. In its simplest form, this conjecture asks: What additional conditions are necessary on a sequence of Riemannian manifolds with positive scalar curvature so that a subsequence converges in the Sormani-Wenger Intrinsic Flat sense to a metric space with a notion of positive scalar curvature? One can see a precise conjecture stated by C. Sormani  in \cite{IAS}. J. Park, W. Tian, and C. Wang \cite{PTW} consider rotationally symmetric warped products,   C. Sormani, W. Tian, and C. Wang  studied $\mathbb{S}^1\times_f \mathbb{S}^2$ warped products with an important example already explored \cite{WCS} and $W^{1,p}$, $p < 2$ convergence obtained \cite{tian2023compactness}. B. Allen, W. Tian, and C. Wang \cite{allen-tian-wang} studied the conformal case and showed a similar convergence as Dong \cite{Dong-PMT_Stability} for the stability of the positive mass theorem. An overarching goal of the work of this paper is to develop an extensive understanding of what conditions are needed on a sequence of Reimannian manifolds in order to conclude compactness in the sense of uniform, Gromov-Hausdorff, Sormani-Wenger Intrisic Flat, and possibly other notions of convergence. In other words, to develop an analogous theorem as the VADB theorem for compactness. By following the same approach as the development of the VADB theorem, we start by studying the problem for warped product length spaces where many examples can be explored which sharpen our intuition for the general setting.

We now define warped product length spaces, a particular example of which is warped product Riemannian manifolds. If $t_0,t_1 \in \R$, $t_0< t_1$, $f:[t_0,t_1]\rightarrow (0,\infty)$ is a bounded, measurable function, $(\Sigma,\sigma)$ is a compact Riemannian manifold, $M=[t_0,t_1]\times \Sigma$, then we define the length of a piecewise smooth curve $\gamma:[0,1]\rightarrow M$, $\gamma(t)=(x(t),\alpha(t))$ with respect to the warping function $f$ to be
\begin{align}\label{def-Warped Product Length Formula}
    L_f(\gamma)&= \int_0^1 \sqrt{x'(t)^2+f(x(t))^2\sigma(\alpha'(t),\alpha'(t))}dt.
\end{align}
Once one has defined a way of measuring the lengths of curves we can use this to define a distance function in a standard way for $p,q \in M$
\begin{align}\label{def-WeightedDistance}
    d_f(p,q)=\inf\{L_f(\gamma): \gamma \text{ piecewise smooth}, \gamma(0)=p,\gamma(1)=q\}.
\end{align}
We now state our first main theorem where we are able to establish uniform Lipschitz bounds on a sequence of warped product length spaces which leads to compactness of their distance functions in the uniform sense. 

\begin{thm}\label{thm-Main Thm 1} 
     Let $0<c<1 \le C < \infty$, $f_n:[t_0,t_1]\rightarrow (0,\infty)$ be a sequence of bounded, measurable functions, and $Q_n \subset [t_0,t_1]$ a sequence of countable dense subsets. If 
    \begin{align}
        f_n(q) &\le C, \quad \forall q \in Q_n,
        \\ f_n(x) &\ge c, \quad \forall x \in [t_0,t_1], 
    \end{align}
    then $d_{f_n}$ satisfies a Lipschitz bound
    \begin{align}
     c d(p,q)\le  d_{f_n}(p,q) \le \sqrt{2}C d(p,q),
    \end{align}
    for all $p,q \in M$ and a subsequence of $d_{f_n}$ uniformly converges to a metric $d_{\infty}$ with the same topology as the sequence where 
    \begin{align}
     c d(p,q)\le  d_{\infty}(p,q) \le \sqrt{2}C d(p,q).
    \end{align}
\end{thm}

In Example \ref{ex-s_n Converge Quotient Metric Space}, Example \ref{ex-h_n Converge Quotient Metric Space}, and Remark \ref{rmrk-z_n Converge Quotient Metric}, we see that if there exists even just one point where the sequence of warping functions does not satisfy a uniform positive lower bound, then one cannot expect the conclusion of Theorem \ref{thm-Main Thm 1} to be true. This is why it is necessary to assume a uniform lower bound on the sequence of warping functions. In Example \ref{ex-h_n Converge Quotient Metric Space}, we see that even if the sequence of warping functions pointwise converges to $1$ everywhere, we cannot expect compactness as in Theorem \ref{thm-Main Thm 1}. Shortcut examples like this first appeared in the work of B. Allen and C. Sormani \cite{Allen-Sormani, Allen-Sormani-2} where we note that Example \ref{ex-h_n Converge Quotient Metric Space} demonstrates new shortcut behaviour. B. Allen \cite{Allen} has also established H\"{o}lder bounds on a general sequence of Riemannian manifolds, but the technique used there cannot produce a Lipschitz bound.

In Example \ref{ex-k_n uniform convergence} of Section \ref{sec-Examples}, we see an example which does not satisfy the hypotheses of Theorem \ref{thm-Main Thm 1}, and we show in Theorem \ref{thm-k_n no Lipschitz Bound} that Example \ref{ex-k_n uniform convergence} does not satisfy a Lipschitz bound. This shows that we cannot expect to make a weaker assumption and derive the same conclusion as Theorem \ref{thm-Main Thm 1}. That being said, Example \ref{ex-k_n uniform convergence} does uniformly converge to the Euclidean metric. This motivates the pursuit of a different hypothesis which may not imply a uniform Lipschitz bound for the sequence but will imply compactness. In order to state the condition on the sequence of warping functions, we make the following definition.

\begin{defn}\label{def-Dense Countable Family of Subsets Introduction}
    Let $C_n$ be a sequence of non-negative real numbers so that $C_n \rightarrow 0$ as $n \rightarrow \infty$. Let $I_n\subset [t_0,t_1]$ be a sequence of subsets so that $\forall x \in [t_0,t_1]$, $\exists y \in I_n$ so that $|x-y|\le C_n$. We call such a family of subsets a \textbf{dense countable family of subsets}.
\end{defn}

The idea is to only require that the sequence of functions is bounded on a dense countable family of subsets in order to obtain what we call an almost Lipschitz bound. The almost Lipschitz bound is enough to show that the sequence of functions is eventually equicontinuous, defined in Definition \ref{defn-Eventually Equicontinuous Metrics} of Section \ref{sec-Background}, which is enough to prove an Arzel\`{a}-Ascoli type theorem to give compactness. This leads to our second main theorem.

\begin{thm}\label{thm-Main Thm 2} 
     Let $0 < c < 1 \le C < \infty$, $f_n:[t_0,t_1]\rightarrow (0,\infty)$ be a sequence of bounded, measurable functions, $C_n>0$ a sequence so that $C_n \rightarrow 0$ as $n \rightarrow \infty$, and $Q_n \subset [t_0,t_1]$ a dense countable family of subsets. If 
    \begin{align}
        f_n(q) &\le C, \quad \forall q \in Q_n,
        \\ f_n(x) &\ge c, \quad \forall x \in [t_0,t_1],
    \end{align}
    then $d_{f_n}$ satisfies almost Lipschitz bounds,
    \begin{align}
       c d(p,q) \le d_{f_n}(p,q) \le \sqrt{2}C d(p,q)+2C_n,
    \end{align}
    and  a subsequence of $d_{f_n}$ uniformly converges to a metric $d_{\infty}$ with the same topology as the sequence where 
    \begin{align}
     c d(p,q)\le  d_{\infty}(p,q) \le \sqrt{2}C d(p,q).
    \end{align}
\end{thm}
Again, we remind the reader that Example \ref{ex-s_n Converge Quotient Metric Space}, and Example \ref{ex-h_n Converge Quotient Metric Space} show that we cannot make a weaker assumption on the lower bound of the sequence of warping functions and Example \ref{ex-v_n Converges to a Metic with Different Topology} shows that if we do not have bounds on a dense countable family of subsets, then we also cannot expect the conclusions of Theorem \ref{thm-Main Thm 2} to hold. So the hypotheses of Theorem \ref{thm-Main Thm 2} are necessary for the desired conclusion.

In Section \ref{sec-Background}, we provide background on warped product lengths spaces, uniform convergence, Lipschitz bounds, the Arzel\`{a}-Ascoli theorem, and what we call almost Lipschitz bounds with important consequences. Many important foundational results are established in this section and examples of metric spaces needed throughout the rest of the paper are described in detail.

In Section \ref{sec-Examples}, we provide many examples of sequences of warped product length spaces which serve to build intuition for the reader as well as justify and explain the hypotheses given in the two main theorems. In particular, the examples show that one cannot assume weaker hypotheses and obtain the same conclusions of the two main theorems of this paper.

In Section \ref{sec-Main Proofs}, we give the proofs of the main theorems.

\textbf{Acknowledgements:} This research was conducted by Lehman college undergraduate math majors under the direction of Assistant Professor Brian Allen. We would like to thank Rafael Gonzalez, the Mathematics and Statistics Student Success Center - MS3, and the S-Stem grant for funding Yaihara Torres for the summer. We would also like to thank the PSC-CUNY Award System for awarding Brian Allen a Traditional B grant which funded Bryan Sanchez for the summer. We would also like to thank the STEM-IN grant program in the school of Natural and Social Sciences at Lehman for funding Brian Allen for the summer.

\section{Background}\label{sec-Background}
Here, we remind the reader of material necessary to understand the remainder of the paper.

\subsection{Warped Product Length Spaces}

Here, we want to consider $(\Sigma,\sigma)$ a smooth, compact Riemannian manifold and $M=[t_0,t_1]\times \Sigma$, $t_0, t_1 \in \R$, $t_0 < t_1$. Let $\alpha:[0,1]\rightarrow \Sigma$ be a piecewise smooth curve and define a piecewise smooth curve $\gamma:[0,1]\rightarrow M$ given by $\gamma(t)=(x(t),\alpha(t))$ where $x:[0,1] \rightarrow [t_0,t_1]$. For points $p \in M$, we define the notation $p=(x(p),p_{\Sigma})$ where $x(p) \in [t_0,t_1]$ and $p_{\Sigma} \in \Sigma$. 

If $f:[t_0,t_1]\rightarrow (0,\infty)$ is a bounded, measurable function, we define the length of $\gamma$ with respect to the warping function $f$ to be
\begin{align}\label{def-Warped Product Length Formula}
    L_f(\gamma)&= \int_0^1 \sqrt{x'(t)^2+f(x(t))^2\sigma(\alpha'(t),\alpha'(t))}dt.
\end{align}
Notice that the length is defined when $f$ is bounded and measurable and we do not have to require that $f$ is smooth or continuous. 

Once one has defined a way of measuring the lengths of curves we can use this to define a distance function in a standard way for $p,q \in M$
\begin{align}\label{def-WeightedDistance}
    d_f(p,q)=\inf\{L_f(\gamma): \gamma \text{ piecewise smooth}, \gamma(0)=p,\gamma(1)=q\}.
\end{align}

We also define some important distance functions for $p,q \in M$ (See \cite{BBI})
\begin{align}
    d(p,q)&=d_1(p,q)=\sqrt{|x(q)-x(p)|^2+d_{\sigma}(p_{\Sigma},q_{\Sigma})^2},\label{def-Generalized Euclidean Distance}
    \\d_{taxi}(p,q)&= |x(q)-x(p)|+d_{\sigma}(p_{\Sigma},q_{\Sigma}),\label{def-Generalized Taxi Distance}
\end{align}
where $d_{\sigma}$ is the standard distance function defined on the smooth Riemannian manifold $(\Sigma, \sigma)$. We also remind the reader of the relationship between these two standard metric spaces
\begin{align}\label{eq-Taxi to Euclidean}
    d(p,q) \le d_{taxi}(p,q) \le \sqrt{2} d(p,q), \quad \forall p,q \in M.
\end{align}

We also define the standard product metric space on $M \times M$ to be
\begin{align}\label{def-Standard Product Metric M times M}
    d_{M\times M}((p,q),(p',q'))=\sqrt{d(p,p')^2+d(q,q')^2},
\end{align}
where we note that
\begin{align}\label{Eq-Product Metric to Product Taxi Inequality}
   d_{M\times M}((p,q),(p',q')) \le d(p,p')+d(q,q'). 
\end{align}

Here, we can generalize the notion of a line connecting two points $p,q \in M$ and define $\ell_{pq}:[0,1]\rightarrow M$ as
\begin{align}\label{def-Generalized Line}
    \ell_{pq}(t)=((x(q)-x(p))t+x(p), \alpha(t)),
\end{align}
where $\alpha:[0,1]\rightarrow \Sigma$ is the distance realizing curve connecting $p_{\Sigma}$ to $q_{\Sigma}$ inside $(\Sigma, \sigma)$. We will often take advantage of these generalized lines in Section \ref{sec-Examples} and Section \ref{sec-Main Proofs}.

We now give a straight forward estimate for comparing warped product length space distance functions which we will use frequently.

\begin{thm}\label{thm-Distance Lower Bound Estimate}
    Let $f,g:[t_0,t_1]\rightarrow (0,\infty)$ be functions so that $f(x) \ge g(x)$ $\forall x \in [t_0,t_1]$ then
    \begin{align}
        d_g(p,q) \le d_f(p,q),
    \end{align}
    for any $p,q \in M$.
\end{thm}
\begin{proof}
   First we notice
   \begin{align}
       L_f(\gamma)&= \int_0^1 \sqrt{x'(t)^2+f(x(t))^2y'(t)^2}dt
       \\&\ge\int_0^1 \sqrt{x'(t)^2+g(x(t))^2y'(t)^2}dt=L_g(\gamma).
   \end{align}
   Since this is true for all piecewise smooth $\gamma$ connecting $p$ to $q$ we see that
   \begin{align}
       d_f(p,q) &=\inf\{L_f(\gamma): \gamma \text{ piecewise smooth}, \gamma_n(0)=p,\gamma_n(1)=q\}
       \\&\ge \inf\{L_g(\gamma): \gamma \text{ piecewise smooth}, \gamma_n(0)=p,\gamma_n(1)=q\}=d_g(p,q).
   \end{align}
\end{proof}

We also note that if one allows $f:[t_0,t_1]\rightarrow [0,\infty)$ then one can still obtain a well defined degenerate metric space where the function will not be positive definite but will be non-negative, symmetric, and satisfy the triangle inequality. As an example of degenerate warped product length space, we give the following degenerate example which will be used frequently in Section \ref{sec-Examples}.

\begin{ex}\label{ex-Quotient Metric Space}
Consider the warping function   
\begin{align}
    s_{\infty}=
    \begin{cases}
       0 & x=0
       \\ 1 & 0<x\le 1
    \end{cases}  ,  
    \end{align}
    which defines a degenerate metric $d_{s_{\infty}}$. We can then define a metric  $\hat{d}_{s_{\infty}}=d_{s_{\infty}}/\sim$ on the set $[0,1]^2$ where we define $P=\{(0,y):y \in [0,1]\}$ to be one point. Then
    \begin{align}
        \hat{d}_{s_{\infty}}(p,q)=\min\{L_{s_{\infty}}(\ell_{pq}), L_{s_\infty}(\beta)\},
    \end{align}
    where $\ell_{pq}$ is the straight line between $p=(x_1,y_1)$ and $q=(x_2,y_2)$ and 
    \begin{align}
        \beta(t)=
        \begin{cases}
           (x_1(1-t),y_1)&t \in [0,1]
           \\ (0,(y_2-y_1)(t-1)+y_1)&t\in[1,2]
           \\ (x_2(t-2),y_2)&t \in [2,3]
        \end{cases}.
    \end{align}
\end{ex}
\begin{proof}
    Let $p,q \in [0,1]^2$ and $\gamma:[0,1]\rightarrow [0,1]^2$ be a piecewise smooth curve connecting $p$ to $q$. We will proceed by considering two cases.

\textbf{Case 1: }$\gamma \cap P \not = \emptyset$

Consider the path $\gamma (t) = (x(t), y(t))$ $\forall t \in [0,1]$. Then let $[a, b] \in [0,1]$ such that $a \in [0,1]$ is the first time so that $\gamma_n(a) \in P$ and $b \in [0,1]$ is the last time such that $\gamma_n(b)\in P$. Let $p=(x_1,y_1)$, $r=(0,y_1)$, $s=(0,y_2)$, and $q=(x_2,y_2)$ and note that $d(p,r)=x_1$ and $d(s,q)=x_2$. Then we can calculate
    \begin{align}
        L_{s_\infty}(\gamma)
        &= L_{s_\infty}(\gamma |_{[0,a]}) + L_{s_\infty}(\gamma |_{[a,b]}) + L_{s_\infty}(\gamma |_{[b,1]})
        \\&=\int_0^a \sqrt{x'(t)^2 +y'(t)^2} dt + \int_a^b \sqrt{x'(t)^2 +s_{\infty}(x(t)) y'(t)^2} dt 
        \\&\quad + \int_b^1 \sqrt{x'(t)^2 +y'(t)^2} dt
        \\&\ge\int_0^a \sqrt{x'(t)^2 +y'(t)^2} dt + \int_b^1 \sqrt{x'(t)^2 +y'(t)^2} dt
        \\&\geq d(p,r) + d(s,q)
        \\&=x_1 + x_2=L_{s_\infty}(\beta).
    \end{align}
Thus, $\beta (t)$ is the shortest path for any path $\gamma (t) \cap P \neq \emptyset$.

\textbf{Case 2:} $\gamma \cap P  = \emptyset$

Since $s_\infty(x) = 1$ for $x \in (0,1]$, we know that the length of $\gamma$ with respect to $s_{\infty}$ is given by the Euclidean length. Thus, the straight line $\ell_{pq}$ connecting $p$ to $q$ is the shortest path in this class 
    \begin{align}
        L_{s_\infty}(\gamma) \ge L_{s_\infty}(\ell _{pq}).
    \end{align}
    
As a result of Case 1 and Case 2, we notice 
\begin{align}
    \hat{d}_{s_\infty}(p,q) &= \inf\{L_{s_{\infty}}(\gamma):\gamma(0)=p,\gamma(1)=q\} 
    \\&\ge \min\{L_{s_\infty}({\ell_{pq})}, L_{s_\infty}(\beta)\} \ge \hat{d}_{s_\infty}(p,q),
\end{align} 
and hence $ \hat{d}_{s_\infty}(p,q) =\min\{L_{s_\infty}({\ell_{pq})}, L_{s_\infty}(\beta)\})$.
\end{proof}

Now we notice that if one allows $f:[t_0,t_1]\rightarrow (0,\infty]$, then one can still obtain a well defined metric space which will not have the same topology of a warped product length space defined with respect to a piecewise continuous non-negative warping function. We note that where the warping function is infinity we are not allowed to travel in any direction besides horizontally. As an example of a metric space of this type, we give the following example which will be used in Section \ref{sec-Examples}.

\begin{ex}\label{ex-Blow Up Example Metric Description}
    If one defines a metric with respect to the warping function
    \begin{align}
    v_{\infty}(x)=
    \begin{cases}
        \infty & 0\le x < \frac{1}{2}
        \\ 1 & \frac{1}{2}\le x \le 1
    \end{cases},
\end{align}
then if $p,q \in [0,1]^2$, $p=(x_1,y_1)$, and $q=(x_2,y_2)$ we can describe the distance function on $[0,1]^2$ as
\begin{align}
d_{v_{\infty}}(p,q)&=
    \begin{cases}
        |x_1-\frac{1}{2}|+|y_2-y_1|+|x_2-\frac{1}{2}| & \text{ if } x_1,x_2< \frac{1}{2}, y_1\not = y_2
       \\ |x_2-x_1| & \text{ if } x_1,x_2< \frac{1}{2}, y_1 = y_2
        \\ |x_1-\frac{1}{2}|+d((\frac{1}{2},y_1),q) & \text{ if } x_1< \frac{1}{2}, x_2\ge \frac{1}{2}
        \\ |x_2-\frac{1}{2}|+d(p,(\frac{1}{2},y_2)) & \text{ if } x_1\ge \frac{1}{2}, x_2<\frac{1}{2}
        \\ d(p,q) & \text{ if } x_1,x_2\ge \frac{1}{2}
    \end{cases},
\end{align}
which does not have the same topology as the Euclidean metric or any other warped product length space with a piecewise continuous and non-negative warping function.
\end{ex}
\begin{proof}
        Let $p=(x_1,y_1)$ and $q=(x_2,y_2)$ and consider the following cases in order to determine an explicit expression for the distance function.

    \textbf{Case 1:} $x_1,x_2 \ge \frac{1}{2}$

    In this case $p$ and $q$ are contained in the part of $[0,1]^2$ where the warping factor is Euclidean and so the shortest path between two points is the straight line distance and hence $d_{v_{\infty}}(p,q)=d(p,q)$, as desired.

    \textbf{Case 2:} $x_1 < \frac{1}{2}$ or $x_2< \frac{1}{2}$

    Hence, any curve $\gamma$ connecting $p$ to $q$ must travel through the region of $[0,1]^2$ where $x<\frac{1}{2}$. By definition of the warping factor, if there is any component of $\gamma$ which travels in the vertical direction, then the length will be infinite. So the only curve which could possibly not be infinite is a curve which travels purely horizontally.

    \textbf{Case 3:} $x_1,x_2< \frac{1}{2}$ and $y_1=y_2$ 
    Here, the horizontal line will have finite length and be the shortest curve which implies $d_{v_{\infty}}(p,q)=|x_2-x_1|$.

    \textbf{Case 4:} $x_1,x_2< \frac{1}{2}$ and $y_1\not=y_2$ 
    
    Then, in order to define a curve with finite length, one must travel horizontally from $p$ and $q$ until one reaches the region in $[0,1]^2$ where $x \ge \frac{1}{2}$. Then, the shortest path from $(\frac{1}{2},y_1)$ to $(\frac{1}{2},y_2)$ is the horizontal line with length $|y_2-y_1|$ and the entire distance is given by $d_{v_{\infty}}(p,q)=|x_1-\frac{1}{2}|+|y_2-y_1|+|x_2-\frac{1}{2}|$.

    \textbf{Case 5:} $x_1< \frac{1}{2}$ and $x_2 \ge \frac{1}{2}$, or $x_1\ge \frac{1}{2}$ and $x_2 < \frac{1}{2}$ 
    
    Here, one must travel horizontally first until one reaches the region inside $[0,1]^2$ where $x \ge \frac{1}{2}$ and then travel along the straight line by Case 1. So one obtains either $d_{v_{\infty}}(p,q)=|x_1-\frac{1}{2}|+d((\frac{1}{2},y_1),q)$ or $d_{v_{\infty}}(p,q)=|x_2-\frac{1}{2}|+d(p,(\frac{1}{2},y_2))$.

    In order to see that $d_{v_{\infty}}$ defines a different topology then Euclidean space, we can consider $p=(\frac{1}{4},\frac{1}{4})$ and notice that the open ball of radius $\frac{1}{8}$ around $p$ with respect to $d_{v_{\infty}}$ is $\{(x,y) \in [0,1]^2: y=\frac{1}{4}, \frac{3}{8} < x < \frac{5}{8}\}$ which is not open in the topology defined by Euclidean space.
\end{proof}

\subsection{Uniform Convergence}

First, we remind the reader of the definition of uniform convergence for a sequence of distance functions.

\begin{defn}\label{defn-Uniform Convergence}
    Let $d_j:X \times X \rightarrow (0,\infty)$ be a sequence of metrics and $d_0:X \times X\rightarrow (0,\infty)$ a metric. Then we say that $d_j \rightarrow d_0$ \textbf{uniformly} if 
    \begin{align}
        \sup_{p,q \in X} \left|d_j(p,q)-d_0(p,q) \right|\rightarrow 0.
    \end{align}
\end{defn}

\begin{rmrk}
    The definition of uniform convergence given in Definition \ref{defn-Uniform Convergence} is equivalent to the definition of convergence for a sequence in the metric space $(\mathcal{C},d_{sup})$ if we let $\mathcal{C}$ be the set of continuous functions from $X \times X \rightarrow (0,\infty)$ and for $f,g \in \mathcal{C}$ we define the distance between points to be 
    \begin{align}
        d_{sup}(f,g)= \sup_{p,q \in X}|f(p,q)-g(p,g)|.
    \end{align}
\end{rmrk}
We now prove a routine squeeze theorem for uniform convergence of distance functions which will be used throughout the paper.

\begin{thm}\label{thm-Squeeze}
     Let $d_j:X \times X \rightarrow (0,\infty)$ be a sequence of metrics and $d_0:X \times X\rightarrow (0,\infty)$ a metric. If for all $p,q \in X$ we find that
     \begin{align}\label{eq-squeeze equation}
         d_0(p,q) - C_j \le d_j(p,q) \le d_0(p,q)+C_j,
     \end{align}
     where $C_j$ is a sequence of real numbers so that $C_j \rightarrow 0$ then $d_j \rightarrow d_0$ uniformly.
\end{thm} 
\begin{proof}
By rearranging \eqref{eq-squeeze equation}, we find $\forall p,q \in [0,1]^2$ 
  \begin{align}\label{eq-squeeze equation 2}
          - C_j \le d_j(p,q)-d_0(p,q) \le C_j,
     \end{align}
     which can be further rearranged so that
     \begin{align}\label{eq-squeeze equation 3}
          0\le |d_j(p,q)-d_0(p,q)| \le C_j.
     \end{align}
     Since \eqref{eq-squeeze equation 3} is true $\forall p,q \in [0,1]^2$ we find
     \begin{align}
           0\le\sup\{ |d_j(p,q)-d_0(p,q)|:p,q \in [0,1]^2\} \le C_j.
     \end{align}
     Hence, by the squeeze theorem for sequences, we find
     \begin{align}
           \sup\{ |d_j(p,q)-d_0(p,q)|:p,q \in [0,1]^2\} \rightarrow 0,
     \end{align}
     which implies that $d_j \rightarrow d_0$ uniformly by Definition \ref{defn-Uniform Convergence}
\end{proof}

\subsection{Lipschitz Bounds and the Arzel\`{a}-Ascoli Theorem}
We first review the definitions required to state the Arzel\`{a}-Ascoli Theorem.

\begin{defn}\label{def-Equicontinuity}
    We say that a sequence of functions $f_j:[0,1]\rightarrow (0,\infty)$ defining metric spaces $d_{f_j}:[0,1]^2\times [0,1]^2 \rightarrow (0,\infty)$ is \textbf{equicontinuous} if  $\forall \varepsilon>0$ $\exists \delta>0$ s.t $\forall (p,q),(p',q') \in M\times M$ $\forall j \in \N$ if $d_{M\times M}((p,q),(p',q'))< \delta$ then 
    \begin{align}
        |d_{f_j}(p,q)-d_{f_j}(p',q')|<\varepsilon.
    \end{align}
\end{defn}

We now show that a Lipschitz bound implies equicontinuity.

 \begin{thm}\label{thm-Lipschitz Implies Equicontinuity}
    Let $f_j:[t_0,t_1]\rightarrow (0,\infty)$ be a sequence of functions defining metric spaces $d_{f_j}:M\times M \rightarrow (0,\infty)$. If $\exists C>0$ so that
    \begin{align}
     d_{f_j}(p,q) \le C d(p,q), \forall p,q \in M, \quad \forall j \in \N,
    \end{align}
    then $d_{f_j}$ is equicontinuous.
\end{thm}
\begin{proof}
 Let $\varepsilon>0$ given and we will make a choice for $\delta>0$. First note that by  \eqref{def-Standard Product Metric M times M} we see that if $d_{M \times M}((p,q),(p',q')) \le \delta$ then
 \begin{align}
    \delta \ge  \sqrt{ d(p,p')^2+d(q,q')^2} \ge \max\{d(p,p'),d(q,q')\},
 \end{align}
 and hence $d(p,p')\le \delta$ and $d(q,q') \le \delta$.

 Now we calculate
 \begin{align}\label{eq-First Important Observation}
 \begin{split}
     |d_{f_j}(p,q)-d_{f_j}(p',q')| &= |d_{f_j}(p,q)-d_{f_j}(p,q')+d_{f_j}(p,q')-d_{f_j}(p',q')|
     \\&\le |d_{f_j}(p,q)-d_{f_j}(p,q')|+|d_{f_j}(p,q')-d_{f_j}(p',q')|.     
 \end{split}
 \end{align}
 Now by applying the consequence of the triangle inequality, $|d(p,q)-d(q,q')| \le d(p,q')$ for any metric space, to \eqref{eq-First Important Observation} we see 
  \begin{align}\label{eq-Second Important Observation}
     |d_{f_j}(p,q)-d_{f_j}(p',q')| &\le d_{f_j}(q,q')+d_{f_j}(p,p').
 \end{align}
 Hence, if we choose $\delta=\frac{\varepsilon}{2C}$ we see by \eqref{eq-Second Important Observation} and the hypothesized distance comparison that
  \begin{align}\label{eq-Second Important Observation}
     |d_{f_j}(p,q)-d_{f_j}(p',q')| &\le Cd(q,q')+C d(p,p') \le C\frac{\varepsilon}{2C}+C\frac{\varepsilon}{2C}=\varepsilon.
 \end{align}
\end{proof}

We now remind the reader of the Arzel\`{a}-Ascoli Theorem which we will aim to apply in order to prove Theorem \ref{thm-Main Thm 1}.

\begin{thm}[Arzel\`{a}-Ascoli Theorem]\label{thm-Arzela Ascoli Theorem}
   Let $d_{f_j}$ be a sequence of bounded and equicontinuous metrics defined on $M$. Then $d_{f_j}$ has a uniformly converging subsequence.
\end{thm}

\subsection{Almost Lipschitz Bounds and Eventually Equicontinuous}
In this subsection we want to define a notion that we call eventually equicontinuous for a sequence. We note that the intuition for this definition comes from the observation that we should only need equicontinuity for the tail of a sequence when proving the Arzela-Ascol\'{i} Theorem. 

\begin{defn}\label{defn-Eventually Equicontinuous Metrics}
  We say that a sequence of functions $f_j:[t_0,t_1]\rightarrow (0,\infty)$ defining metric spaces $d_{f_j}:M\times M \rightarrow (0,\infty)$ is \textbf{eventually equicontinuous} if  $\forall \varepsilon>0$ $\exists \delta>0, N \in \N$ s.t $\forall (p,q),(p',q') \in M\times M$ $\forall j \ge N$ if $d_{M \times M}((p,q),(p',q'))< \delta$ then 
    \begin{align}
        |d_{f_j}(p,q)-d_{f_j}(p',q')|<\varepsilon.
    \end{align}  
\end{defn}

We find it interesting to note that the Arzel\`{a}-Ascoli Theorem still holds under the weaker assumption of eventually equicontinuous. This does not seem to be surprising but since we will use it later we think it is important to make note of it now. Our aim is to apply Theorem \ref{thm- Arzela Ascoli Theorem 2} in order to prove Theorem \ref{thm-Main Thm 2}.

\begin{thm}\label{thm- Arzela Ascoli Theorem 2}
   Let $d_{f_j}$ be a sequence of bounded and eventually equicontinuous metrics defined on $M$. Then $d_{f_j}$ has a uniformly converging subsequence.
\end{thm}
\begin{proof}
    One can check that the usual proof of the Arzel\`{a}-Ascoli Theorem works in exactly the same way here.
\end{proof}

We now show that being almost Lipschitz, as demonstrated in the next Theorem, is enough to imply eventually equicontinuous. 

 \begin{thm}\label{thm-Almost Lipschitz Implies Eventually Equicontinuity}
    Let $f_j:[0,1]\rightarrow (0,\infty)$ be a sequence of functions defining metric spaces $d_{f_j}:[0,1]^2\times [0,1]^2 \rightarrow (0,\infty)$. If $\exists C, C_n>0$ so that
    \begin{align}
     d_{f_j}(p,q) \le C d(p,q)+C_n, \forall p,q \in M, \quad \forall j \in \N,
    \end{align}
    where $C_n\rightarrow 0$ as $n\rightarrow \infty$, then $d_{f_j}$ is eventually equicontinuous.
\end{thm}
\begin{proof}
     Let $\varepsilon>0$ given and we will make a choice for $\delta>0$. First note that by  \eqref{def-Standard Product Metric M times M} we see that if $d_{M \times M}((p,q),(p',q')) \le \delta$ then
 \begin{align}
    \delta \ge  \sqrt{ d(p,p')^2+d(q,q')^2} \ge \max\{d(p,p'),d(q,q')\},
 \end{align}
 and hence $d(p,p')\le \delta$ and $d(q,q') \le \delta$.

 Now we calculate
 \begin{align}\label{eq-First Important Observation 2}
 \begin{split}
     |d_{f_j}(p,q)-d_{f_j}(p',q')| &= |d_{f_j}(p,q)-d_{f_j}(p,q')+d_{f_j}(p,q')-d_{f_j}(p',q')|
     \\&\le |d_{f_j}(p,q)-d_{f_j}(p,q')|+|d_{f_j}(p,q')-d_{f_j}(p',q')|.     
 \end{split}
 \end{align}
 Now by applying the consequence of the triangle inequality, $|d(p,q)-d(q,q')| \le d(p,q')$ for any metric space, to \eqref{eq-First Important Observation 2} we see 
  \begin{align}\label{eq-Second Important Observation}
     |d_{f_j}(p,q)-d_{f_j}(p',q')| &\le d_{f_j}(q,q')+d_{f_j}(p,p').
 \end{align}
 Hence, if we choose $\delta=\frac{\varepsilon}{4C}$ and $N \in \N$ so that $C_n \le \frac{\varepsilon}{4}$ for all $n \ge N$, we see by \eqref{eq-Second Important Observation} and the hypothesized distance comparison that
  \begin{align}\label{eq-Second Important Observation}
     |d_{f_j}(p,q)-d_{f_j}(p',q')| &\le Cd(q,q')+C_n+C d(p,p') +C_n 
     \\&\le C\frac{\varepsilon}{4C}+\frac{\varepsilon}{4}+C\frac{\varepsilon}{4C}+\frac{\varepsilon}{4}=\varepsilon.
 \end{align}
\end{proof}

\section{Examples}\label{sec-Examples}
It was observed by B. Allen and C. Sormani \cite{Allen-Sormani, Allen-Sormani-2} that one will need to impose stronger conditions on a sequence of Riemannian metrics to control the sequence of distance functions from below than to control the sequence of distance functions from above. In \cite{Allen-Sormani, Allen-Sormani-2} the authors were trying to prove convergence to specific Riemannian manifolds in the limit but the insight gained also applies to the case where one is trying to prove compactness. Here we will review similar examples and give some new examples which will explain all of the hypotheses made in the main theorems. All examples in this section will be given where $\Sigma=[0,1]$, $\sigma=dx^2$, and $M=[0,1]^2$. 

\subsection{Shortcut Examples}

We begin by reviewing shortcut examples which elucidate the fact that we will need to impose uniform, positive lower bounds on the sequence of warping functions in order to ensure that any convergent subsequence of distance functions will converge to a positive definite distance function with the same topology as the sequence.

\begin{ex}\label{ex-s_n Converge Quotient Metric Space} 
Define the sequence of warping functions
\begin{align}
    s_n(x)=
    \begin{cases}
        \frac{1}{n}& 0\le x \le \frac{1}{n}
        \\ 1 & \frac{1}{n}< x \le 1
    \end{cases}
\end{align}
  and consider the sequence of metric spaces,  $d_{s_n}$. Then $d_{s_n}$ converges uniformly to the metric space with warping function
    \begin{align}
    s_{\infty}=
    \begin{cases}
       0 & x=0
       \\ 1 & 0<x\le 1
    \end{cases}    
    \end{align}
    where we define a metric $\hat{d}_{s_{\infty}}=d_{s_{\infty}}/\sim$ where we define $P=\{(0,y):y \in [0,1]\}$ to be one point.
\end{ex}
Since the proof of this example was basically already given by B. Allen and C. Sormani in \cite{Allen-Sormani, Allen-Sormani-2} we omit the proof. We also note that the proof is very similar to the proof of the next example which we do include. The significance of the next example is that $h_n \rightarrow 1$ pointwise everywhere but yet $d_{h_n}$ does not even pointwise almost everywhere converge to $d$. This shows that pointwise convergence of the warping functions is not enough to establish any control from below and hence we will be forced to assume uniform bounds from below in the main theorems.

\begin{ex}\label{ex-h_n Converge Quotient Metric Space}
    If we define
    \begin{align}
    h_n(x)=
    \begin{cases}
    1 & 0\le x < \frac{1}{n}
       \\ \frac{1}{n}& \frac{1}{n}\le x \le \frac{2}{n}
        \\ 1 & \frac{2}{n}< x \le 1
    \end{cases},
\end{align}
then the sequence of metric spaces $d_{h_n}$ converges uniformly to the metric space with warping function
    \begin{align}
    s_{\infty}=
    \begin{cases}
       0 & x=0
       \\ 1 & 0<x\le 1
    \end{cases},    
    \end{align}
    where we define a metric  $\hat{d}_{s_{\infty}}=d_{s_{\infty}}/\sim$  on the set $[0,1]^2$ where we define $P=\{(0,y):y \in [0,1]\}$ to be one point.
\end{ex}
\begin{proof}
We start by showing that $d_{h_n}(p,q) \le d_{s_{\infty}}(p,q)+C_n$ where $C_n \rightarrow 0$ and $n \rightarrow \infty$.

\textbf{Case 1: $x_1,x_2 > 0$}

Choose $N \in \N$ large enough such that $\forall n \ge N$ we find $\frac{2}{n} < \min\{{x_1,x_2}\}$. So, if the line $\ell_{pq}$ connects $p$ to $q$, we find
\begin{align} \label{h2sl1above}
    d_{h_{n}}(p,q) \le L_{h_{n}}(\ell_{pq}) = L_{s_{\infty}}(\ell_{pq}) .
\end{align}

Now if we define
\begin{align}
    \rho_1(t)=
    \begin{cases}
        (x_1 +(\frac{2}{n}-x_1)t,y_1)&t\in[0,1]
        \\(\frac{2}{n}, y_1 +(y_2-y_1)(t-1) &t\in[1,2]
        \\(\frac{2}{n}+(x_2 -\frac{2}{n})(t-2), y_2) &t\in[2,3]
    \end{cases},
\end{align}
we can calculate
\begin{align} \label{h2sb1above}
    d_{h_{n}}(p,q) &\le L_{h_{n}}(\rho_1)
    \\&= |x_1 - \frac{2}{n}| + \frac{1}{n}|y_2-y_1| + |x_2 - \frac{2}{n}|
    \\&\le x_1 + x_2 + \frac{1}{n}|y_2-y_1|
    \\&\le L_{s_{\infty}}(\beta) + \frac{1}{n},
\end{align}
where $\beta$ is defined in Example \ref{ex-Quotient Metric Space}.

By combining \eqref{h2sl1above} and \eqref{h2sb1above}, we see
\begin{align}
    d_{h_{n}}(p,q) &\le \min\{L_{s_{\infty}}(\ell_{pq}), L_{s_{\infty}}(\beta) + \frac{1}{n}\}
    \\&\le \min\{L_{s_{\infty}}(\ell_{pq}), L_{s_{\infty}}(\beta)\} + \frac{1}{n}
    \\&= \hat{d}_{s_{\infty}}(p,q) + \frac{1}{n} ,\label{h2d1above}
\end{align}
where we applied Example \ref{ex-Quotient Metric Space} in \eqref{h2d1above}.

\textbf{Case 2: $x_1 = 0$ and $x_2 > 0$}

Choose $N \in \N$ large enough such that $\forall n \ge N$ we find $\frac{2}{n} < x_2$.
Consider the path $\rho_1:[0,2]\rightarrow [0,1]^2$ connecting points $p = (0,y_1)$ and $q = (x_2,y_2)$
    \begin{align}
        \rho_1(t)=
    \begin{cases}
        (\frac{2}{n}t,y_1) &t\in [0,1]
        \\(\frac{2}{n}+(x_2-\frac{2}{n})(t-1),y_1 + (y_2 - y_1)(t-1)) &t\in[1,2]    
    \end{cases}
    \end{align}

Now, we can calculate
\begin{align}
  \label{d2sl2above} d_{h_{n}}(p,q) & \le L_{h_{n}}(\rho_1)
   \\&= \frac{2}{n} + \sqrt{\left|x_2- \frac{2}{n}\right|^2 + \left| y_2 - y_1\right|^2}
   \\&\le  L_{s_{\infty}}(\ell_{pq}) + \frac{2}{n} + \left|\sqrt{\left|x_2 - \frac{2}{n}\right|^2 + |y_2 - y_1|^2} - \sqrt{|x_2|^2 + |y_2 - y_1|^2}\right|
   \\&= L_{s_{\infty}}(\ell_{pq}) + A_n,
\end{align}
where $A_n = \frac{2}{n} + \left|\sqrt{\left|x_2 - \frac{2}{n}\right|^2 + |y_2 - y_1|^2} - \sqrt{|x_2|^2 + |y_2 - y_1|^2}\right|$.

Now, if we define the curve $\rho_2:[0,3]\rightarrow [0,1]^2$ by
\begin{align}
    \rho_2(t)=
    \begin{cases}
        (\frac{2}{n}t,y_1)&t\in[0,1]
        \\(\frac{2}{n}, y_1 +(y_2-y_1)(t-1) &t\in[1,2]
        \\(\frac{2}{n}+(x_2 -\frac{2}{n})(t-2), y_2) &t\in[2,3]
    \end{cases},
\end{align}
then we can calculate
\begin{align}
  \label{d2sb2above}  d_{h_{n}}(p,q) &\le L_{h_n}(\rho_2)
    \\&= \frac{2}{n} + \frac{1}{n}|y_2 - y_1| + x_2-\frac{2}{n}
    \\&=\frac{1}{n}|y_2 - y_1| + x_2
    \\&= L_{s_{\infty}}(\beta) + \frac{1}{n}|y_2 - y_1|
    \le L_{s_{\infty}}(\beta) + \frac{1}{n},
\end{align}
where $\beta$ is defined in Example \ref{ex-Quotient Metric Space}.

By combining \eqref{d2sl2above} and \eqref{d2sb2above}, we find
\begin{align}
     d_{h_{n}}(p,q) &\le \min\{L_{s_{\infty}}(\ell_{pq}) + A_n, L_{s_{\infty}}(\beta) + \frac{1}{n}\}
     \\& \le \min\{L_{s_{\infty}}(\ell_{pq})  + \max\{A_n, \frac{1}{n}\}, L_{s_{\infty}}(\beta) + \max\{A_n, \frac{1}{n}\}\}
     \\& \le \min\{L_{s_{\infty}}(\ell_{pq}), L_{s_{\infty}}(\beta)\} + \max\{A_n, \frac{1}{n}\}
     \\& = \hat{d}_{s_{\infty}}(p,q) + \max\left\{A_n, \frac{1}{n}\right\},
\end{align}
which completes this case.

\textbf{Case 3: $x_1 = x_2 = 0$}

If we define the curve $\rho_3:[0,3]\rightarrow [0,1]^2$ by
\begin{align}
    \rho_3(t)=
    \begin{cases}
        (\frac{1}{n}t,y_1)&t\in[0,1]
        \\(\frac{1}{n}, y_1 +(y_2-y_1)(t-1) &t\in[1,2]
        \\(\frac{1}{n}-\frac{1}{n}(t-2), y_2) &t\in[2,3]
    \end{cases},
\end{align}
then we can calculate
\begin{align}
    0=d_{s_{\infty}}(p,q) &\le d_{h_n}(p,q) 
    \\&\le L_{h_n}(\rho_3) 
    \\&=\frac{1}{n}+ \frac{1}{n}|y_2-y_1|+\frac{1}{n}
    = \frac{2}{n}+\frac{1}{n}|y_2-y_1|+d_{s_{\infty}}(p,q),
\end{align}
which completes this case.

Now we want to show that $d_{s_{\infty}}(p,q)-C_n \le d_{h_n}(p,q) $ where $C_n \rightarrow 0$ and $n \rightarrow \infty$. We will also consider cases for this estimate where Case 3 above already applies here.

\textbf{Case 1:} $x_1,x_2 >0$

Let $\gamma:[0,1]\rightarrow [0,1]^2$, $\gamma_n(t)=(x(t),y(t))$ be a curve connecting $p$ to $q$.

If $\gamma_n([0,1]) \cap P = \emptyset$ then there exists $N \in \N$ so that for $n \ge N$ we find that $\gamma_n(t) \ge \frac{2}{n}$ for all $t \in [0,1]$ and hence
\begin{align}\label{eq: First Time for inf h_n}
    L_{h_n}(\gamma) = L_{s_{\infty}}(\gamma) \ge d_{s_{\infty}}(p,q).
\end{align}

If $\gamma_n([0,1]) \cap P \not = \emptyset$ then we can decompose $[0,1]$ into three subintervals $[0,a_n]$, $(a_n,b_n)$, $[b_n,1]$ where $a_n$ is the first time where $\gamma_n(a_n)=(2/n,y)$ for some $y \in (0,1)$ and $b_n$ is the last time $\gamma_n(b_n)=(2/n,y)$ for some $y \in [0,1]$. Now we can estimate the length
\begin{align}\label{eq-5678}
\begin{split}
    L_{h_n}(\gamma)&= L_{h_n}(\gamma|_{[0,a_n]})+ L_{h_n}(\gamma|_{(a_n,b_n)})+ L_{h_n}(\gamma|_{[b_n,1]})  
    \\&\ge L(\gamma|_{[0,a_n]})+ L(\gamma|_{[b_n,1]}).       
\end{split}
\end{align}

 Since we know that the Euclidean length of a curve connecting $p$ to $\gamma_n(a_n)$ is longer than the straight line connecting those points, say $\ell_{p\gamma_n(a_n)}$, and the same thing for $\gamma_n(b_n)$ to $q$ via $\ell_{\gamma_n(b_n)q}$ we find
\begin{align}\label{eq-9101112}
L(\gamma|_{[0,a_n]})+ L(\gamma|_{[b_n,1]}) \ge L(\ell_{p\gamma_n(a_n)})+ L(\ell_{\gamma_n(a_n)q}).
\end{align}
Since a horizontal line is the shortest way to connect $p$ and $q$ to $\{(2/n,y):y \in (0,1)\}$ we can combine \eqref{eq-5678} and \eqref{eq-9101112} to further estimate
\begin{align}\label{eq: Second Time for inf h_n }
    L_{h_n}(\gamma)&\ge L(\gamma|_{[0,a_n]})+ L(\gamma|_{[b_n,1]}) 
    \\&\ge |x_1-2/n|+|x_2-2/n|
    \\&\ge |x_1|+|x_2|-4/n 
    \\&= L_{s_{\infty}}(\beta)-4/n\ge d_{s_{\infty}}(p,q)-4/n,\label{eq: Second Time for inf h_n End }
\end{align}
where $\beta$ is defined in Example \ref{ex-Quotient Metric Space}.

Now by taking the infimum over all curves $\gamma$ connecting $p$ to $q$ and using \eqref{eq: First Time for inf h_n} and \eqref{eq: Second Time for inf h_n }-\eqref{eq: Second Time for inf h_n End } we find
\begin{align}
    d_{h_n}(p,q) \ge d_{s_{\infty}}(p,q)-4/n,
\end{align}
which completes this case.

\textbf{Case 2:} $x_1=0$ and $x_2>0$

Let $\gamma:[0,1]\rightarrow [0,1]^2$, $\gamma_n(t)=(x(t),y(t))$ be a curve connecting $p$ to $q$ and decompose $[0,1]$ into two subintervals $[0,a_n)$ and $[a_n,1]$ where $a_n \in [0,1]$ is the last time where $\gamma_n(a_n)=(2/n,y)$ for some $y \in [0,1]$. Then we calculate
\begin{align}\label{eq-First h_n Observation}
    L_{h_n}(\gamma)&= L_{h_n}(\gamma|_{[0,a_n)})+L_{h_n}(\gamma|_{[a_n,1]}) \ge L(\gamma|_{[a_n,1]}).
\end{align}
 Since we know that the Euclidean length of any curve connecting two endpoints is larger than the straight line connecting the same endpoints we know that 
\begin{align}\label{eq-Second h_n Observation}
  L(\gamma|_{[a_n,1]}) \ge L(\ell_{\gamma_n(a_n)q})
\end{align}
where $\ell_{\gamma_n(a_n)q}$ is the straight line connecting $\gamma_n(a_n)$ to $q$. Hence, by combining \eqref{eq-First h_n Observation} with \eqref{eq-Second h_n Observation} we find
\begin{align}
    L_{h_n}(\gamma)&\ge L(\ell_{\gamma_n(a_n)q}).
\end{align}
Since a horizontal line is the shortest way to connect $q$ to $\{(2/n,y):y \in (0,1)\}$, we further estimate
\begin{align}
    L_{h_n}(\gamma)& \ge L(\ell_{\gamma_n(a_n)q}) \label{eq:h2ds2below}
    \\&\ge |x_2-2/n|\label{eq-Needs More Justification}
    \\&\ge |x_2|-\frac{2}{n} 
   \\&= L_{s_{\infty}}(\beta)-\frac{2}{n} 
    \ge d_{s_{\infty}}(p,q)-\frac{2}{n},\label{eq:h2ds2belowEnd}
\end{align}
where $\beta$ is defined in Example \ref{ex-Quotient Metric Space}.
Now by taking the infimum over all $\gamma$ connecting $p$ to $q$ in \eqref{eq:h2ds2below}-\eqref{eq:h2ds2belowEnd}, we find
\begin{align}
    d_{h_n}(p,q) \ge d_{s_{\infty}}(p,q)-\frac{2}{n},
\end{align}
which completes this case

Putting everything together we obtain the estimate
\begin{align}
    d_{s_{\infty}}(p,q)-\frac{1}{n} \le d_{h_n}(p,q) \le d_{s_{\infty}}(p,q)+C_n,
\end{align}
where $C_n \rightarrow 0$ as $n \rightarrow \infty$. Now by Theorem \ref{thm-Squeeze}, we see the desired convergence of $d_{h_n}$ to $d_{s_{\infty}}$ which can be made into the metric $\hat{d}_{s_{\infty}}$ by identifying $P$ to be one point.
\end{proof}
\begin{rmrk}\label{rmrk-z_n Converge Quotient Metric}
    One could also consider the sequence of functions defined by
    \begin{align}
    z_n(x)=
    \begin{cases}
        \frac{1}{n}& x=0
        \\ 1 & 0< x \le 1
    \end{cases},
\end{align}
 and show that $d_{z_n}$ converges uniformly to $\hat{d}_{s_{\infty}}$. This shows that if even just one point $x \in [0,1]$ is such that $f_n(x)$ does not satisfy a uniformly positive lower bound then one should not expect the sequence of distance functions to converge to a positive definite distance function. Again, this is why we need to assume a uniform lower bound in all of the main theorems.
\end{rmrk}

\subsection{Blow-Up Examples}

In this subsection we will see that there is much more flexibility in the conditions on a sequence of warping functions required to control a sequence of warped product length spaces from above. In particular, we see that any blow up rate will be allowed as long as the blow up occurs on an increasingly small interval. 

\begin{ex}\label{ex-k_n uniform convergence}
    For the sequence of functions defined by \begin{align}
    k_n(x)=
    \begin{cases}
        n^{\alpha}& 0\le x < \frac{1}{n}
        \\ 1 & \frac{1}{n}\le x \le 1
    \end{cases},
\end{align}
where $\alpha >0$ we find that $d_{k_n}$ converges uniformly to $d$ given in \eqref{def-Generalized Euclidean Distance}.
\end{ex}
\begin{proof}
    First we notice the $k_n(x) \ge 1$ for all $x\in[0,1]$. So by Theorem \ref{thm-Distance Lower Bound Estimate} we know
    \begin{align} 
       d(p,q) \le d_{k_n}(p, q), \forall p,q \in [0,1]^2.
    \end{align}
    Assume $p=(x_1,y_1)$ and $q=(x_2,y_2)$.
    
\textbf{Case 1:} $x_1,x_2 >0$

Hence, we can choose $N \in \N$ large enough so that for all $n \ge N$, we find $\frac{1}{n} <\min\{x_1,x_2\}$. So if $\ell_{pq}$ is the line connecting $p$ to $q$ we find
\begin{align}
    L_{k_n}(\ell_{pq})=L(\ell_{pq})=d(p,q).
\end{align}

\textbf{Case 2:} $x_1=0$ and $x_2>0$

Consider the path $\gamma (t)$ connecting points $p=(0,y_1)$ and $q=(x_2, y_2)$
    \begin{align}
        \gamma (t)=
    \begin{cases}
        (\frac{t}{n}, y_1) & t \in [0,1]
        \\(\frac{1}{n} +(x_2 - \frac{1}{n})(t-1), y_1 + (y_2 - y_1)(t-1)) &t \in [1, 2] 
    \end{cases}.
    \end{align}
Consider the approximation
    \begin{align}
        L_{k_n} (\gamma) &= \int_0^1 \frac{1}{n}dt + \int_1^2 \sqrt{\left(x_2 - \frac{1}{n}\right)^2 + (y_2 - y_1)^2)} dt
        \\&=\frac{1}{n} + \sqrt{\left(x_2 - \frac{1}{n}\right)^2 + {(y_2 - y_1)^2}}
    \end{align}
Notice that $L_{k_n}(\gamma) \rightarrow d(p,q)$ for $d(p,q)=\sqrt{(x_2 - 0)^2 + (y_2 - y_1)^2}$.
Now we can calculate
      \begin{align}
       d(p,q)&\le d_{k_n}(p,q)
       \\& \le L_{k_n}(\gamma)
       \\&=\frac{1}{n} + \sqrt{\left(x_2 - \frac{1}{n}\right)^2 + (y_2 - y_1)^2}
       \\& =d(p,q)+\frac{1}{n} + \sqrt{\left(x_2 - \frac{1}{n}\right)^2 + (y_2 - y_1)^2} - \sqrt{(x_2 - 0)^2 + (y_2 - y_1)^2)}
       \\& \le d(p,q)+\frac{1}{n} + \left|\sqrt{\left(x_2 - \frac{1}{n}\right)^2 + (y_2 - y_1)^2} - \sqrt{(x_2 - 0)^2 + (y_2 - y_1)^2)}\right|
       \\& \le d(p,q)+C_n,
    \end{align}
    where $C_n \rightarrow 0$ as $n \rightarrow \infty$.

\textbf{Case 3:} $x_1=x_2=0$\\
Consider the path $\gamma (t)$ connecting points $p=(0,y_1)$ and $q=(0, y_2)$
    \begin{align}
        \gamma (t)=
    \begin{cases}
        (\frac{t}{n}, y_1) & t \in [0,1]
        \\  (\frac{1}{n},y_1+(y_2-y_1)(t-1)),&t \in [1, 2] 
        \\(\frac{1}{n}(3-t),y_2) &t \in [2, 3] 
    \end{cases}.
    \end{align}
Now by calculating the length of $\gamma$, we find
\begin{align}
   d(p,q) &\le d_{k_n}(p,q)
   \\&\le L_{k_n}(\gamma)
    \\&=\frac{1}{n}+(y_2-y_1)+\frac{1}{n}
    \\&=d(p,q)+\frac{2}{n}.
\end{align}
By combining all three cases, we see
\begin{align}
    d(p,q) \le d_{k_n}(p,q) \le d(p,q) + \max\{C_n,\frac{2}{n}\},
\end{align}
and hence, we can apply Theorem \ref{thm-Squeeze} to obtain the result.
\end{proof}

Now we will see that the previous example will not satisfy a Lipschitz bound. A similar example was studied by B. Allen and E. Bryden in Example 3.1 of \cite{Allen-Bryden-First} where it was shown that no uniform Lipschitz bound held, but a uniform H\"{o}lder bound did hold. Here, we emphasize that our proof works for any blow-up rate quantified by $\alpha>0$.

\begin{thm}\label{thm-k_n no Lipschitz Bound}
    For the sequence of functions defined by \begin{align}
    k_n(x)=
    \begin{cases}
        n^{\alpha}& 0\le x < \frac{1}{n}
        \\ 1 & \frac{1}{n}\le x \le 1
    \end{cases},
\end{align}
where $\alpha >0$ then show that there exists $p_n,q_n \in [0,1]^2$ and a sequence $C_n \rightarrow \infty$ so that
\begin{align}
     d_{k_n}(p_n,q_n) \ge C_n d(p_n,q_n).
\end{align}
\end{thm}
\begin{proof}
First, we will prove the following claim:
Let $p_n=(0,y_{1_n})$, $q_n=(0,y_{2_n})$, and $|y_{2_n}-y_{1_n}| = \frac{2}{n^{\alpha+1}(n^{\alpha}-1)}$, then $d_{k_n}(p_n,q_n)\ge L_{k_n}(\ell_{p_n q_n})$.

Notice that for $n \ge 2$
\begin{align}
    \frac{2}{n} \ge \frac{2}{n(n^{\alpha}-1)}  =\frac{2n^{\alpha}}{n^{\alpha+1}(n^{\alpha}-1)}=L_{k_n}(\ell_{pq}). \label{lpq2n}
\end{align}
Let $\gamma_n:[0,1]\rightarrow [0,1]^2$, $\gamma_n(t)=(x_n(t),y_n(t))$ be a curve connecting $p_n$ to $q_n$. Notice that if $\gamma_n(t) \cap \{
(\frac{1}{n}, y) : y \in [0,1]\} \neq \emptyset$ and we let $a,b \in [0,1]$, $a < b$ so that $a$ is the first time $\gamma_n$ enters the $\{
(\frac{1}{n}, y) : y \in [0,1]\}$ region and $b$ is the last time. Then we can calculate 
\begin{align}
    L_{k_n}(\gamma_n(t)) &= \int_0^a \sqrt{x'_n(t)^2+k_n(t)^2y'_n(t)^2} dt +\int_a^b \sqrt{x'_n(t)^2+y'_n(t)^2} dt
    \\&\quad +\int_b^1\sqrt{x'_n(t)^2+k_n(t)^2y'_n(t)^2} dt
    \\&\ge \int_0^a |x'_n(t)| dt +\int_a^b |y'_n(t)| dt + \int_b^1 |x'_n(t)| dt
    \\&\ge \left|\int_0^a x'_n(t) dt\right| + \left|\int_a^b y'_n(t) dt\right| + \left|\int_b^1 x'_n(t)dt\right|
    \\&=|0-\frac{1}{n}|+ |y_{2_n} - y_{1_n}| +|\frac{1}{n}-0|
    \\&\geq \frac{2}{n} + |y_{2_n} - y_{1_n}| \geq \frac{2}{n}. \label{Lkngamma}
\end{align}

Now, by combining \eqref{lpq2n} and \eqref{Lkngamma}, we have for $\gamma_n(t) \cap \{
(\frac{1}{n}, y) : y \in [0,1]\} \neq \emptyset$ that
\begin{align}
    L_{k_n}(\gamma_n(t)) \geq L_{k_n}(\ell_{p_nq_n}).
\end{align}

Then for $\gamma_n(t) \cap \{
 (\frac{1}{n}, y) : y \in [0,1]\} = \emptyset$ we find
\begin{align}
    L_{k_n}(\gamma_n(t))&=\int_0^1 \sqrt{x'_n(t)^2+k_n(t)^2y'_n(t)^2} dt 
    \\&\geq \int_0^1 |k_n(x)y'_n(t)| dt
    \\&= n^{\alpha}\left| \int_0^1  y'_n(t)dt \right|
    \\&=n^{\alpha}|y_{1_n}-y_{2_n}|
    \\&=\frac{2n^\alpha}{n^{\alpha+1}(n^{\alpha}-1)}=L_{k_n}(\ell_{p_nq_n}),
\end{align}
so $\ell_{p_nq_n}$ will always be the shortest path for this sequence of functions. Hence, by taking the infimum over all curves connecting $p_n$ to $q_n$, we have proved the claim.

Now consider the sequence of points $p_n=(0, y_{1_n})$ and $q_n=(0,y_{2_n})$ such that $|y_{2_n}-y_{1_n}| =\frac{2}{n^{\alpha+1}(n^{\alpha}-1)}$ and let $C_n = n^\alpha$. Then we can calculate
    \begin{align}
    d_{k_n}(p_n,q_n)&\ge L_{k_n}(\ell_{p_{n}q_{n}})
    =n^\alpha|y_{2_n}-y_{1_n}|
    =n^\alpha d(p_n,q_n)
     = C_n d(p_n,q_n).
    \end{align}
This completes the proof showing there exists a sequence of points $p_n,q_n \in [0,1]^2$ and a sequence $C_n \rightarrow \infty$ so that $d_{k_n}(p_n,q_n) \geq C_nd(p_n,q_n)$.
\end{proof}

Now, we see that if the blow up occurs at only one point, then we can obtain a Lipschitz bound. Theorem \ref{thm-k_n no Lipschitz Bound} and Example \ref{ex-w_n uniform convergence} explain why we need the hypothesis of dense bounds for the entire sequence in order to obtain the conclusion of Theorem \ref{thm-Main Thm 1}.

\begin{ex}\label{ex-w_n uniform convergence}
    For the sequence of functions defined by
    \begin{align}
    w_n(x)=
    \begin{cases}
        n^{\alpha}& x=0
        \\ 1 & 0< x \le 1
    \end{cases}
\end{align}
where $\alpha >0$, we find that $d_{w_n}$ converges uniformly to $d$, the Euclidean distance function, and 
\begin{align}
    d_{w_n}(p,q) \le d(p,q).
\end{align}
\end{ex}
\begin{proof}
Notice that $1 \le w_n \le k_n$ and hence by Theorem \ref{thm-Distance Lower Bound Estimate} we know
\begin{align}
    d(p,q) \le d_{w_n}(p,q) \le d_{k_n}(p,q), \quad \forall p,q \in [0,1]^2,
\end{align}
where $k_n$ is defined in Example \ref{ex-k_n uniform convergence} where we also proved that $d_{k_n} \rightarrow d$ uniformly. Hence
\begin{align}
    d(p,q) \le d_{w_n}(p,q) \le d(p,q)+C_n, \quad \forall p,q \in [0,1]^2,
\end{align}
where $C_n \rightarrow 0$. So by Theorem \ref{thm-Squeeze}, we see that $d_{w_n}\rightarrow d$ uniformly.

Now we want to show that $d_{w_n}$ satisfies a Lipschitz bound.

\textbf{Case 1: $x_1=x_2=0$}

Pick a $x_\varepsilon \in (0,1]$ $\forall\varepsilon>0$ such that $|x_\varepsilon-x_1|<\varepsilon$. Now, define $p_\varepsilon =(x_\varepsilon, y_1)$ and $q_\varepsilon = (x_\varepsilon, y_2)$. We can calculate
    \begin{align}
        d_{w_n}(p,q)&\leq d_{w_n}(p,p_\varepsilon) + d_{w_n}(p_\varepsilon, q_\varepsilon) +d_{w_n}(q_\varepsilon,q)
        \\& \le L_{w_n}(\ell_{pp_{\varepsilon}})+L_{w_n}(\ell_{p_{\varepsilon}q_{\varepsilon}})+L_{w_n}(\ell_{q_{\varepsilon}q})
        \\&= \varepsilon + |y_1-y_2| + \varepsilon
        \\&=2\varepsilon +d(p,q).
    \end{align}
Since this is true $\forall \varepsilon >0$, we find that $d_{w_n}(p,q)\leq d(p,q)$.

\textbf{Case 2: $x_1\in (0,1]$, and $x_2\in [0,1]$} 

Notice
    \begin{align}
        d_{w_n}(p,q) \le L_{w_n}(\ell_{pq})= L(\ell_{pq}) = d(p,q).
    \end{align}
Hence the result follows.
\end{proof}

Now we see that if the blow up occurs on an interval which is not vanishing along the sequence then the sequence of distance functions will converge to a metric which does not have the topology of a warped product length space defined with respect to a non-negative and piecewise continuous warping function. We note that there is nothing special about the choice of $[0,\frac{1}{2})$ in this example and if one built an example which blew up on any fixed interval a similar behaviour would appear. Thus demonstrating the importance of at least having a dense countable family of subsets where the sequence of warping functions is bounded in Theorem \ref{thm-Main Thm 2}.

\begin{ex}\label{ex-v_n Converges to a Metic with Different Topology}
    For the sequence of functions defined by \begin{align}
    v_n(x)=
    \begin{cases}
        n^{\alpha}& 0\le x < \frac{1}{2}
        \\ 1 & \frac{1}{2}\le x \le 1
    \end{cases},
\end{align}
where $\alpha >0$, we find that $d_{v_n}$ converges uniformly to the metric space $d_{v_{\infty}}$, described carefully in Example \ref{ex-Blow Up Example Metric Description}, which does not have the same topology as the sequence.
\end{ex}
\begin{proof}
We start by establishing an upper bound $d_{v_n} \le d_{v_{\infty}}$, by considering a few cases. Let $p=(x_1,y_1)$ and $q=(x_2,y_2)$.

\textbf{Case 1:} $x_1,x_2 \ge \frac{1}{2}$.

We calculate
\begin{align}
    d_{v_n}(p,q) \le L_{v_n}(\ell_{pq}) =L(\ell_{pq}) = d(p,q)=d_{v_{\infty}}(p,q),
\end{align}
in this case.

\textbf{Case 2:} $x_1 < \frac{1}{2}$ and $x_2\ge \frac{1}{2}$ or $x_1 \ge \frac{1}{2}$ and $x_2< \frac{1}{2}$.

Without loss of generality, let $x_1 < \frac{1}{2}$ and $x_2 \geq \frac{1}{2}$, and 
\begin{align}
    \gamma(t) =
    \begin{cases}
        ((\frac{1}{2}-x_1)t + x_1, y_1) \quad &t \in [0,1]
        \\((x_2-\frac{1}{2})(t-1) + \frac{1}{2}, (y_2-y_1)(t-1) + y_1) \quad &t \in [1,2]
    \end{cases}.
\end{align}
Then, we can calculate
\begin{align}
    d_{v_n}(p,q) &\le L_{v_n}(\gamma(t)) 
    \\&= \left|x_1-\frac{1}{2}\right| + \sqrt{\left|x_2 - \frac{1}{2}\right|^2 + |y_2 - y_1|^2} 
    \\&= \left|x_1 - \frac{1}{2}\right| + d\left(\left(\frac{1}{2},y_1\right),q\right) =d_{v_{\infty}}(p,q),
\end{align}
in this case.

\textbf{Case 3:} $x_1 < \frac{1}{2}$, $x_2< \frac{1}{2}$, and $y_1\not = y_2$.

Let 
\begin{align}
    \gamma(t)= 
    \begin{cases}
        ((\frac{1}{2}-x_1)t + x_1, y_1) \quad &t \in [0,1]
        \\(\frac{1}{2}, (y_2-y_1)(t-1) + y_1) \quad &t \in [1,2]
        \\((x_2-\frac{1}{2})(t-2) + \frac{1}{2}, y_2) \quad &t \in [2,3]
    \end{cases}.
\end{align}
Then, we can calculate
\begin{align}
    d_{v_n}(p,q) \le L_{v_n}(\gamma(t)) = \left|x_1-\frac{1}{2}\right| + |y_2 - y_1| + \left|x_2 - \frac{1}{2}\right| =d_{v_{\infty}}(p,q),
\end{align}
in this case.

\textbf{Case 4:} $x_1 < \frac{1}{2}$, $x_2< \frac{1}{2}$, and $y_1 = y_2$.

We can calculate
\begin{align}
    d_{v_n}(p,q) \le L_{v_n}(\ell_{pq}) =L(\ell_{pq}) = |x_2 - x_1| =d_{v_{\infty}}(p,q),
\end{align}
in this case.

Hence, we have established the desired upper bound on $d_{v_n}$. Now, we establish a lower bound $d_{v_n} \ge d_{v_{\infty}}-C_n$ by again considering the same four cases.

Let $\gamma_n(t)=(x_n(t),y_n(t))$ be any curve connecting $p$ to $q$.

 \textbf{Case 1:} $x_1,x_2 \ge \frac{1}{2}$.

 Since $1 \le v_n$ we know by Theorem \ref{thm-Distance Lower Bound Estimate} that $d_{v_n}(p,q) \ge d(p,q)=d_{v_{\infty}}(p,q)$ in this case. 

\textbf{Case 2:} $x_1 < \frac{1}{2}$ and $x_2\ge \frac{1}{2}$ or $x_1 \ge \frac{1}{2}$ and $x_2< \frac{1}{2}$.

Without loss of generality we can just consider the first situation where $x_1 < \frac{1}{2}$ and $x_2\ge \frac{1}{2}$.

Then if $a \in [0,1]$ is the last time that $x_n(a)=\frac{1}{2}$, then we can calculate
\begin{align}
     L_{v_n}(\gamma_n)&=\int_0^a \sqrt{x'_n(t)^2+v_n(x_n(t))^2y'_n(t)^2}dt+\int_a^1 \sqrt{x'_n(t)^2+y'_n(t)^2}dt
     \\&\geq \int_0^a |x'_n(t)| +\int_a^1 \sqrt{x'_n(t)^2+y'_n(t)^2}dt
     \\&= \left|\int_0^a x'_n(t)\right| +\int_a^1 \sqrt{x'_n(t)^2+y'_n(t)^2}dt
     \\&= \left|x_1 - \frac{1}{2}\right| + d(\gamma_n(a),q).\label{eq-Basically What we want to Show}
\end{align}
Now, we want to show that  $\gamma_n(a)\rightarrow (\frac{1}{2},y_1)$ for a curve which is attempting to minimize the distance, i.e. 
\begin{align}
    d_{v_n}(p,q) \ge L_{v_n}(\gamma_n)-\frac{1}{n}.\label{eq-Almost Minimizing Curve}
\end{align}

Notice that $\gamma_n\cap \{(\frac{1}{2},y):y\in[0,1]\}\not = \emptyset$ and let $a' \in (0,1]$ be the first time $\gamma_n$ intersects $(\frac{1}{2},y)$. Then we find
 \begin{align}
    L_{v_n}(\gamma_n)&= \int_0^{a'}\sqrt{x'_n(t)^2+v_n(x_n(t))^2y'_n(t)^2}dt+\int_{a'}^a\sqrt{x'_n(t)^2+v_n(x_n(t))^2y'_n(t)^2}dt
    \\&\quad+\int_a^1\sqrt{x'_n(t)^2+v_n(x_n(t))^2y'_n(t)^2}dt.
    \end{align}
 Since $v_n$ is constant when $0 \le x <\frac{1}{2}$ and $v_n=1$ when $\frac{1}{2} \le x \le 1$ we find
 \begin{align}
     L_{v_n}(\gamma_n)& \ge L_{v_n}(\ell_{p\gamma_n(a')})+ L_{v_n}(\ell_{\gamma_n(a')\gamma_n(a)})+ L_{v_n}(\ell_{\gamma_n(a)q})
     \\& \ge L_{v_n}(\ell_{p\gamma_n(a')})+ L(\ell_{\gamma_n(a')\gamma_n(a)})+ L(\ell_{\gamma_n(a)q})\label{eq-Point Should Converge}
     \\& \ge L_{v_n}(\ell_{p\gamma_n(a')})+ L(\ell_{\gamma_n(a')q})\label{eq-Point Should Converge 2},
 \end{align}
 where a straight line will be shorter connecting $\gamma_n(a')$ to $q$ in \eqref{eq-Point Should Converge} to \eqref{eq-Point Should Converge 2}.
 Now, notice that $L_{v_n}(\ell_{p\gamma_n(a')}) \rightarrow \infty$ unless $\gamma_n(a')\rightarrow (\frac{1}{2},y_1)$ and by \eqref{eq-Point Should Converge} and \eqref{eq-Point Should Converge 2} we see that $\gamma_n$ would be an even shorter admissible curve if $|\gamma_n(a)-\gamma_n(a')|\rightarrow 0$ and hence $\gamma_n(a) \rightarrow (\frac{1}{2},y_1)$.

 So by \eqref{eq-Basically What we want to Show} and \eqref{eq-Almost Minimizing Curve}, we see that
 \begin{align}
      d_{v_n}(p,q) &\ge\left|x_1 - \frac{1}{2}\right| + d(\gamma_n(a),q) - \frac{1}{n}
      \\&\ge\left|x_1 - \frac{1}{2}\right| + d((\frac{1}{2},y_1),q) - C_n
      \\&\ge d_{v_{\infty}}(p,q) - C_n,
 \end{align}
 where $C_n \rightarrow 0$ as $n \rightarrow \infty$, as desired.

\textbf{Case 3:} $x_1 < \frac{1}{2}$, $x_2< \frac{1}{2}$, and $y_1\not = y_2$.

Notice that if $\gamma \cap (\frac{1}{2},y) = \emptyset$, then we have
    \begin{align}
    L_{v_n}(\gamma)&= \int_0^1\sqrt{x'(t)^2+v_n(x(t))^2y'(t)^2}dt
\ge L_{v_n}(\ell_{pq}),
    \end{align}
 since $v_n$ is constant when $0 \le x <\frac{1}{2}$. Since $v_n(x) = n^{\alpha}$ for $x<\frac{1}{2}$, $L_{v_n}(\ell_{pq}) \rightarrow \infty$ when $y_1\not =y_2$, so we must take a path that intersects $(\frac{1}{2}, y)$ for some $y \in [0,1]$.

Thus, take a path $\gamma$ such that $\gamma \cap (\frac{1}{2},y) \neq \emptyset$. Then if $a \in [0,1]$ is the first time that $x(a)=\frac{1}{2}$ and $b \in [0,1]$ is the last time that $x(b) = \frac{1}{2}$, we can use $v_n(x) = 1$ for $x\geq\frac{1}{2}$ to calculate
\begin{align}
    L_{v_n}(\gamma(t))&=\int_0^a \sqrt{x_n'(t)^2+v_n(x(t))^2y_n'(t)^2}dt +\int_a^b \sqrt{x_n'(t)^2+v_n(x(t))^2y_n'(t)^2}dt  
    \\&\quad + \int_b^1 \sqrt{x_n'(t)^2+v_n(x(t))^2y_n'(t)^2}dt
    \\&\geq \int_0^a |x_n'(t)|dt + \int_a^b |y_n'(t)|dt + \int_b^1 |x_n'(t)|dt
    \\&\geq \left|\int_0^a x_n'(t)dt\right| + \left|\int_a^b y_n'(t)dt\right| + \left|\int_b^1 x_n'(t)dt\right|
    \\&= \left|x_1 - \frac{1}{2}\right| + |y_2 - y_1| + \left|x_2 - \frac{1}{2}\right| = d_{v_\infty}(p,q).
\end{align}

\textbf{Case 4:} $x_1 < \frac{1}{2}$, $x_2< \frac{1}{2}$, and $y_1 = y_2$.
Here we calculate
\begin{align}
    L_{v_n}(\gamma)&= \int_0^1\sqrt{x'(t)^2+v_n(x(t))^2y'(t)^2}dt
    \\&\ge \int_0^1|x'(t)|dt
    \\&\ge \left|\int_0^1x'(t)dt \right|
    \\&=|x_2-x_1|=d_{v_{\infty}}(p,q).
\end{align}
Now we have established the lower bound by considering four cases.

Hence we see by Theorem \ref{thm-Squeeze} that $d_{v_n} \rightarrow d_{v_\infty}$ uniformly. By Example \ref{ex-Blow Up Example Metric Description} we see that $d_{v_{\infty}}$ does not have the same topology as the sequence.
\end{proof}

\section{Proofs of Main Theorems}\label{sec-Main Proofs}

Our first goal is to develop a condition which implies that a sequence of warped product length spaces satisfies a Lipschitz bound when compared to the taxi metric. Since the taxi metric is uniformly equivalent to the Euclidean metric this will imply that the sequence satisfies a Lipschitz bound when compared to Euclidean distance as well.

\begin{thm}\label{thm-Dense Subset Comparison to Taxi}
    Let $C \ge 1$, $f:[t_0,t_1]\rightarrow (0,\infty)$ be a function, and $Q \subset [t_0,t_1]$ a countable dense subset. Then if 
    \begin{align}
        f(q) \le C, \quad \forall q \in Q,
    \end{align}
    then 
    \begin{align}
        d_{f}(p,q) \le C d_{taxi}(p,q),
    \end{align}
    for all $p,q \in M=[t_0,t_1]\times \Sigma$.
\end{thm}
\begin{proof}
We will proceed by considering two cases.

\textbf{Case 1: $x(p) = x(q)$}

Pick a $x_\varepsilon \in Q$ such that $|x_\varepsilon-x(p)|<\varepsilon$ and define $p_{\varepsilon}=(x_{\varepsilon},p_{\Sigma})$, $q_{\varepsilon}=(x_{\varepsilon},q_{\Sigma})$. Now, using the definition of generalized lines for warped products \eqref{def-Generalized Line}, we can calculate
    \begin{align}
        d_f(p,q)&\leq d_f(p,p_\varepsilon) + d_f(p_\varepsilon, q_\varepsilon) +d_f(q_\varepsilon,q)
        \\&\leq L_f(\ell_{pp_\varepsilon})+L_f(\ell_{p_{\varepsilon}q_{\varepsilon}})+L_f(\ell_{q_{\varepsilon}q})
        \\&\leq \varepsilon + Cd_{\sigma}(p_{\Sigma},q_{\Sigma}) + \varepsilon
        \\&=2\varepsilon +Cd_{taxi}(p,q).
    \end{align}
Since this is true for all $\varepsilon$, we find that $d_f(p,q)\leq Cd_{taxi}(p,q)$.

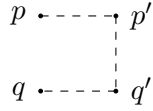
\begin{figure}[h]
 \centering
 \begin{tikzpicture}[scale=1]
  \draw[dashed] (1,2) -- (2,2) [radius=0.025];
  \draw[dashed] (2,2) -- (2,1) [radius=0.025];
  \draw[dashed] (2,1) -- (1,1) [radius=0.025];
   \draw[fill] (1,2) circle [radius=0.025];
   \node[left, outer sep=2pt] at (1,2){$p$};
   \draw[fill] (1,1) circle [radius=0.025];
   \node[left, outer sep=2pt] at (1,1) {$q$};
   \draw[fill] (2,2) circle [radius=0.025];
      \node[right, outer sep=2pt] at (2,2){$p'$};
   \draw[fill] (2,1) circle [radius=0.025];
   \node[right, outer sep=2pt] at (2,1) {$q'$};
 \end{tikzpicture}
\caption{When $x(p) = x(q)$.}
 \label{fig:curve}
\end{figure}

\textbf{Case 2: $x(p) \neq x(q)$}

Without loss of generality, we may assume that $x(p)<x(q)$. Pick a $x'\in Q\cap [x(p),x(q)]$ and let $p' = (x',p_{\Sigma})$ and $q' = (x',q_{\Sigma})$. We can then calculate 
    \begin{align}
        d_f(p,q) &=d_f(p,p') + d_f(p',q') +d_f(q',q)
        \\&\leq L(\ell_{pp'}) + L_f(\ell_{p'q'}) + L(\ell_{q'q})
        \\&= |x(p)-x'|+Cd_{\sigma}(p_{\Sigma},q_{\Sigma})+|x(p)-x(q)|-|x(p)-x'|
        \\&=|x(p)-x(q)|+Cd_{\sigma}(p_{\Sigma},q_{\Sigma})
        \\&\le C|x(p)-x(q)|+Cd_{\sigma}(p_{\Sigma},q_{\Sigma})= Cd_{taxi}(p,q).
    \end{align}

    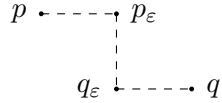
\begin{figure}[h]
 \centering
 \begin{tikzpicture}[scale=1]
  \draw[dashed] (1,2) -- (2,2) [radius=0.025];
  \draw[dashed] (2,2) -- (2,1) [radius=0.025];
  \draw[dashed] (2,1) -- (3,1) [radius=0.025];
   \draw[fill] (1,2) circle [radius=0.025];
   \node[left, outer sep=2pt] at (1,2){$p$};
   \draw[fill] (3,1) circle [radius=0.025];
   \node[right, outer sep=2pt] at (3,1) {$q$};
   \draw[fill] (2,2) circle [radius=0.025];
      \node[right, outer sep=2pt] at (2,2){$p_{\varepsilon}$};
   \draw[fill] (2,1) circle [radius=0.025];
   \node[left, outer sep=2pt] at (2,1) {$q_{\varepsilon}$};
 \end{tikzpicture}
 \caption{When $x(p) \neq x(q)$.}
 \label{fig:curve}
\end{figure}
    
Thus, we reach our conclusion $\forall p,q \in M$, we have $d_f(p,q)\leq Cd_{taxi}(p,q)$.
\end{proof}

We now apply Theorem \ref{thm-Dense Subset Comparison to Taxi} to a sequence to deduce Lipschitz bounds for the sequence.

\begin{cor}\label{cor-Dense Subset Comparison to Taxi} 
     Let $C \ge 1$, $f_n:[t_0,t_1]\rightarrow (0,\infty)$ be a sequence of functions, and $Q_n \subset [t_0,t_1]$ a sequence of countable dense subsets. Then if 
    \begin{align}
        f_n(q) \le C, \quad \forall q \in Q_n,
    \end{align}
    then 
    \begin{align}
        d_{f_n}(p,q) \le C d_{taxi}(p,q),
    \end{align}
    for all $p,q \in M$ and hence $d_{f_n}$ is equicontinuous.
\end{cor}
\begin{proof}
    Apply Theorem \ref{thm-Dense Subset Comparison to Taxi} for each $f_n$ and $Q_n$ to obtain the desired Lipschitz upper bound. The equicontinuity conclusion follows from Theorem \ref{thm-Lipschitz Implies Equicontinuity} combined with \eqref{eq-Taxi to Euclidean}.
\end{proof}

Inspired by Example \ref{ex-k_n uniform convergence} and Theorem \ref{thm-k_n no Lipschitz Bound}, we want to look for a weaker condition which can imply that a sequence is eventually equicontinuous. Notice that Example \ref{ex-k_n uniform convergence} does not satisfy the hypotheses of Theorem \ref{thm-Dense Subset Comparison to Taxi} but, yet, converges uniformly to Euclidean space. This seems to suggest that we should be able to obtain compactness for sequences satisfying a weaker condition than bounded on a countable dense subset, as Theorem \ref{thm-Dense Countable Family Comparison to Taxi} demonstrates. In order to prove the next result we need to recall the definition of a dense countable family of subsets.

\begin{defn}\label{def-Dense Countable Family of Subsets}
    Let $C_n$ be a sequence of non-negative real numbers so that $C_n \rightarrow 0$ as $n \rightarrow \infty$. Let $I_n\subset [t_0,t_1]$ be a sequence of subsets so that $\forall x \in [0,1]$, $\exists y \in I_n$ so that $|x-y|\le C_n$. We call such a family of subsets a \textbf{dense countable family of subsets}.
\end{defn}

Now we use Definition \ref{def-Dense Countable Family of Subsets} to show that a sequence which is bounded on a dense countable family of subsets is almost Lipschitz and, hence, eventually equicontinuous by Theorem \ref{thm-Almost Lipschitz Implies Eventually Equicontinuity}. We note that Example \ref{ex-v_n Converges to a Metic with Different Topology} shows that one cannot prove the conclusion of Theorem \ref{thm-Dense Countable Family Comparison to Taxi} under a weaker hypothesis.

\begin{thm}\label{thm-Dense Countable Family Comparison to Taxi} 
     Let $C \ge 1$, $f_n:[t_0,t_1]\rightarrow (0,\infty)$ be a sequence of functions, and $Q_n \subset [t_0,t_1]$ a dense countable family of subsets. Then if 
    \begin{align}
        f_n(q) \le C, \quad \forall q \in Q_n,
    \end{align}
    then 
    \begin{align}
        d_{f_n}(p,q) \le C d_{taxi}(p,q)+2C_n,
    \end{align}
    for all $p,q \in M$ and hence $d_{f_n}$ is eventually equicontinuous.
\end{thm}
\begin{proof}
    We proceed by considering two cases.

\textbf{Case 1: $x(p) \neq x(q)$.}

Without loss of generality, we may assume that $x(p) < x(q)$. Choose $u_n \in Q_n$ so that $|x_1-u_n|\le C_n$, which can be done by definition of $Q_n$. Define $p=(x(p), p_{\Sigma})$, $q=(x(q), q_{\Sigma})$, $p'=(u_n,p_{\Sigma})$, and $q'=(u_n,q_{\Sigma})$.

Suppose $u_n \in [x(p),x(q)] $ and calculate 
\begin{align}
    d_{f_n}(p,q)&\leq d_{f_n}(p,p') +d_{f_n}(p',q') + d_{f_n}(q',q)
    \\&\leq L_{f_n}(\ell_{pp'})+L_{f_n}(\ell_{p'q'})+L_{f_n}(\ell_{q'q})
    \\&\leq|x(p)-u_n|+f_n(u_n)d_{\sigma}(p_{\Sigma},q_{\Sigma})+|x(q)-x(p)|-|x(p)-u_n|
    \\&=|x(p)-x(q)|+f_n(u_n)d_{\sigma}(p_{\Sigma},q_{\Sigma})
    \\&\leq Cd_{taxi}(p,q).
\end{align}
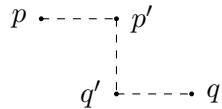
\begin{figure}[h]
 \centering
 \begin{tikzpicture}[scale=1]
  \draw[dashed] (1,2) -- (2,2) [radius=0.025];
  \draw[dashed] (2,2) -- (2,1) [radius=0.025];
  \draw[dashed] (2,1) -- (3,1) [radius=0.025];
   \draw[fill] (1,2) circle [radius=0.025];
   \node[left, outer sep=2pt] at (1,2){$p$};
   \draw[fill] (3,1) circle [radius=0.025];
   \node[right, outer sep=2pt] at (3,1) {$q$};
   \draw[fill] (2,2) circle [radius=0.025];
      \node[right, outer sep=2pt] at (2,2){$p'$};
   \draw[fill] (2,1) circle [radius=0.025];
   \node[left, outer sep=2pt] at (2,1) {$q'$};
 \end{tikzpicture}
 \caption{When $x(p) \neq x(q)$ and $u_n \in [x(p),x(q)]$.}
 \label{fig:curve}
\end{figure}

If $u_n \not \in [x(p),x(q)]$, then we can calculate 
\begin{align}
    d_{f_n}(p,q)&\leq d_{f_n}(p,p') +d_{f_n}(p',q') + d_{f_n}(q',q)
    \\&\leq L_{f_n}(\ell_{pq'})+L_{f_n}(\ell_{p'q'})+L_{f_n}(\ell_{q'q})
    \\&\leq C_n+f_n(u_n)d_{\sigma}(p_{\Sigma},q_{\Sigma})+|x(q)-x(p)|+C_n
    \\&=|x(p)-x(q)|+f_n(u_n)d_{\sigma}(p_{\Sigma},q_{\Sigma}) + 2C_n
    \\&\leq Cd_{taxi}(p,q) + 2C_n 
\end{align}
\begin{figure}[h]
 \centering
 \begin{tikzpicture}[scale=1]
  \draw[dashed] (1,2) -- (0,2) [radius=0.025];
  \draw[dashed] (0,2) -- (0,1) [radius=0.025];
  \draw[dashed] (0,1) -- (2,1) [radius=0.025];
   \draw[fill] (1,2) circle [radius=0.025];
   \node[right, outer sep=2pt] at (1,2){$p$};
   \draw[fill] (2,1) circle [radius=0.025];
   \node[right, outer sep=2pt] at (2,1) {$q$};
   \draw[fill] (0,2) circle [radius=0.025];
      \node[left, outer sep=2pt] at (0,2){$p'$};
   \draw[fill] (0,1) circle [radius=0.025];
   \node[left, outer sep=2pt] at (0,1) {$q'$};
 \end{tikzpicture}
\caption{When $x(p) \neq x(q)$ and $u_n \not \in [x(p),x(q)]$.}
 \label{fig:curve}
\end{figure}

\textbf{Case 2: $x(p) = x(q)$.}

Pick a $u_n \in Q_n$ such that $|x(p)-u_n|<C_n$. Let $p = (x(p),p_{\Sigma})$, $q = (x(q),q_{\Sigma})$, $p' = (u_n,p_{\Sigma})$ and $q' = (u_n,q_{\Sigma})$. Now, we calculate
\begin{align}
    d_{f_n}(p,q) &\leq d_{f_n}(p,p') + d_{f_n}(p',q') + d_{f_n}(q',q)
    \\&\leq L_{f_n}(\ell_{pp'}) + L_{f_n}(\ell_{p' q'}) + L_{f_n}(\ell_{q' q})
    \\&=|x(p)-u_n|+f_n(u_n)d_{\sigma}(p_{\Sigma},q_{\Sigma})+|x(p)-u_n| 
    \\&\leq Cd_{taxi}(p,q)+2C_n.
\end{align}
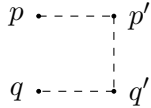
\begin{figure}[h]
 \centering
 \begin{tikzpicture}[scale=1]
  \draw[dashed] (1,2) -- (2,2) [radius=0.025];
  \draw[dashed] (2,2) -- (2,1) [radius=0.025];
  \draw[dashed] (2,1) -- (1,1) [radius=0.025];
   \draw[fill] (1,2) circle [radius=0.025];
   \node[left, outer sep=2pt] at (1,2){$p$};
   \draw[fill] (1,1) circle [radius=0.025];
   \node[left, outer sep=2pt] at (1,1) {$q$};
   \draw[fill] (2,2) circle [radius=0.025];
      \node[right, outer sep=2pt] at (2,2){$p'$};
   \draw[fill] (2,1) circle [radius=0.025];
   \node[right, outer sep=2pt] at (2,1) {$q'$};
 \end{tikzpicture}
\caption{When $x(p) = x(q)$.}
 \label{fig:curve}
\end{figure}

Thus we reach our conclusion $\forall p,q \in M$ we have $d_f(p,q)\leq Cd_{taxi}(p,q)+2C_n$.

\end{proof}

We note that Theorem \ref{thm-Dense Subset Comparison to Taxi} and Theorem \ref{thm-Dense Countable Family Comparison to Taxi} are enough to obtain compactness of the sequence of distance functions, but Example \ref{ex-s_n Converge Quotient Metric Space} and Example \ref{ex-h_n Converge Quotient Metric Space} show that the subsequence will not necessarily converge to a metric space with the same topology as the sequence. Here, we observe the simple fact that a uniform bound from below is enough to ensure that the subsequence will converge to a metric satisfying the same uniform lower bound. We also note that Example \ref{ex-s_n Converge Quotient Metric Space} and Example \ref{ex-h_n Converge Quotient Metric Space} show that we cannot weaken this hypothesis and obtain the same conclusion.

\begin{thm}\label{thm-Bounded warping function}
 Let $C_1>0$ and $f_n:[t_0,t_1]\rightarrow (0,\infty)$ be a sequence of functions. If 
 \begin{align}
     C_1 \le f_n(x) , \quad \forall x \in [t_0,t_1],
 \end{align}
 then
 \begin{align}
\min\{1,C_1\}d(p,q)  \le d_{f_n}(p,q),  
 \end{align}
 for all $p,q \in M$.
\end{thm}
\begin{proof}

Consider any curve $\gamma(t)=(x(t),\alpha(t))$ connecting $p$ to $q$ so that
    \begin{align}
        L_f(\gamma) &= \int_0^1 \sqrt{x'(t)^2+f(x(t))^2\sigma(\alpha',\alpha')}dt
        \\&\ge \min\{1,C_1\} \int_0^1 \sqrt{x'(t)^2+ \sigma(\alpha',\alpha')}dt
        = \min\{1,C_1\}L(\gamma),
    \end{align}
and hence
    \begin{align}
        d_f(p,q) 
        &= \inf\{L_f(\gamma) : \gamma \text{ piecewise smooth}, \gamma(0)=p,\gamma(1)=q\}
       \\ &\geq  \inf\{\min\{1,C_1\}L(\gamma) :\gamma \text{ piecewise smooth}, \gamma(0)=p,\gamma(1)=q\}
       \\& = \min\{1,C_1\}  d(p,q),
    \end{align} 
as desired.
\end{proof}

We now give the proofs of the two main theorems by combining several results established above.

\begin{proof}[Proof of Theorem \ref{thm-Main Thm 1}]
    By Theorem \ref{thm-Dense Subset Comparison to Taxi}, Theorem \ref{thm-Lipschitz Implies Equicontinuity}, and Theorem \ref{thm-Arzela Ascoli Theorem}, we see that a subsequence of the distance functions $d_{f_n}$ uniformly converges to a function $d_{\infty}$ which is symmetric, satisfies the triangle inequality, continuous with respect to $d$, and so that $d_{\infty}\le C d$. Then, Theorem \ref{thm-Bounded warping function} implies that $d_{\infty} \ge c d$ and hence $cd \le d_{\infty} \le C d$. So, $d_{\infty}$ defines a metric with the same topology.
\end{proof}

\begin{proof}[Proof of Theorem \ref{thm-Main Thm 2}]
    By Theorem \ref{thm-Dense Countable Family Comparison to Taxi}, Theorem \ref{thm-Almost Lipschitz Implies Eventually Equicontinuity}, and Theorem \ref{thm- Arzela Ascoli Theorem 2}, we see that a subsequence of the distance functions $d_{f_n}$ uniformly converges to a function $d_{\infty}$ which is symmetric, satisfies the triangle inequality, continuous with respect to $d$, and so that $d_{\infty}\le C d$. Then, Theorem \ref{thm-Bounded warping function} implies that $d_{\infty} \ge c d$ and, hence, $cd \le d_{\infty} \le C d$. So, $d_{\infty}$ defines a metric with the same topology.
\end{proof}

%%%%%%%%%%%%%%%%%%%%%%%%%%%%%%%%%%%%%%%%%%%%%%%%%%%%%%
\bibliographystyle{plain}
\bibliography{bibliography}

\end{document}